\documentclass[a4paper]{amsart}
\numberwithin{equation}{section}
\newtheorem{theorem}{Theorem}[section]
\newtheorem{corollary}{Corollary}[section]
\newtheorem{definition}{Definition}[section]
\newtheorem{lemma}{Lemma}[section]

\theoremstyle{remark}
\newtheorem{remark}{Remark}[section]
\usepackage{hyperref}
\usepackage{graphicx,amsfonts}
\usepackage{color}
\title[Some properties of analytic functions]
 {Some properties of analytic functions related with Booth lemniscate}
\subjclass[2010]{30C45}
\keywords{Univalent, Starlike, Convex, Strongly Starlike, Logarithmic Coefficients, Subordination.}
\begin{document}
\begin{abstract}
The object of the present paper is to study of two certain subclass of analytic functions related with Booth lemniscate which we denote by $\mathcal{BS}(\alpha)$ and $\mathcal{BK}(\alpha)$. Some properties of these subclasses are considered.
\end{abstract}
\author[P. Najmadi, Sh. Najafzadeh and A. Ebadian]
       {P. Najmadi, Sh. Najafzadeh and A. Ebadian }
\address{Department of Mathematics, Payame Noor University, P.O. Box 19395-3697 Tehran, Iran}
       \email {najmadi@phd.pnu.ac.ir {\rm(P. Najmadi)}}
       \email{najafzadeh1234@yahoo.ie {\rm(Sh. Najafzadeh)}}
       \email{aebadian@pnu.ac.ir {\rm (A. Ebadian)}}

\maketitle
\section{Introduction}

Let $\Delta$ be the open unit disk in the complex plane $\mathbb{C}$ and $\mathcal{A}$ be the class of normalized and analytic functions. Easily seen that any $f\in \mathcal{A}$ has the following form:
\begin{equation}\label{f}
  f(z)=z+a_2 z^2+a_3 z^3+\cdots \qquad(z\in \Delta).
\end{equation}
Further, by $\mathcal{S}$ we will denote the class of all functions in $\mathcal{A}$ which are univalent in
$\Delta$. The set of all functions $f\in \mathcal{A}$ that are starlike
univalent in $\Delta$ will be denote by $\mathcal{S}^*$ and the set of all functions $f\in \mathcal{A}$ that are convex univalent in $\Delta$ will be denote by $\mathcal{K}$.
Analytically, the function $f\in \mathcal{A}$ is a starlike univalent function, if and only if
\begin{equation*}
  \mathfrak{Re}\left\{\frac{zf'(z)}{f(z)}\right\}>0\qquad(z\in \Delta).
\end{equation*}
Also, $f\in \mathcal{A}$ belongs to the class $\mathcal{K}$, if and only if
\begin{equation*}
  \mathfrak{Re}\left\{1+\frac{zf''(z)}{f'(z)}\right\}>0\qquad(z\in \Delta).
\end{equation*}
For more details about this functions, the reader may refer to the book of Duren \cite{Dur}. Define by $\mathfrak{B}$ the class of analytic functions $w(z)$ in
$\Delta$ with $w(0)=0$ and $|w(z)|<1$, $(z \in \Delta)$.
Let $f$ and $g$ be two functions in $\mathcal{A}$. Then we say
that $f$ is subordinate to $g$, written $f (z)\prec g(z)$, if there
exists a function $w\in\mathfrak{B}$ such that $f(z)=g(w(z))$ for all
$z\in\Delta$.
Furthermore, if the function $g$ is univalent in $\Delta$, then
we have the following equivalence:
\begin{equation*}
    f (z)\prec g(z) \Leftrightarrow (f (0) = g(0)\quad {\rm and}\quad f (\Delta)\subset g(\Delta)).
\end{equation*}
 Recently, the authors \cite{psok, Pijsok}, (see also \cite{KarEba}) have studied the function
\begin{equation}\label{Falpha}
  F_{\alpha}(z):=\frac{z}{1-\alpha z^2}=\sum_{n=1}^{\infty}\alpha^{n-1}z^{2n-1}\qquad (z\in\Delta,~0\leq \alpha\leq1).
\end{equation}
We remark that the function $F_{\alpha}(z)$ is a starlike univalent function when $0\leq \alpha<1$.
In addition $F_{\alpha}(\Delta)=D(\alpha)$ ($0\leq \alpha<1$), where
\begin{equation*}\label{D(alpha)}
  D(\alpha)=\left\{x+iy\in\mathbb{C}: ~ \left(x^2+y^2\right)^2-\frac{x^2}{(1-\alpha)^2}-\frac{y^2}{(1+\alpha)^2}<0\right\}
\end{equation*}
and
\begin{equation*}\label{D(1)}
  F_1(\Delta)=\mathbb{C}\backslash \left\{(-\infty,-i/2]\cup [i/2,\infty)\right\}.
\end{equation*}

For $f\in \mathcal{A}$ we denote by $Area \, f(\Delta)$, the
area of the multi-sheeted image of the disk $\Delta_r = \{z \in \mathbb{C} : |z| < r\}$ $(0 < r \leq 1)$
under $f$. Thus, in terms of the coefficients of $f$, $f'(z)=\sum_{n=1}^{\infty}n a_n z^{n-1}$ one gets with the help of the classical Parseval-Gutzmer formula (see \cite{Dur}) the
relation
\begin{equation}\label{area}
  Area \, f(\Delta)=\int\int_{\Delta_r}|f'(z)|^2{\rm d}x{\rm d}y=\pi \sum_{n=1}^{\infty}n |a_n|^2 r^{2n},
\end{equation}
which is called the Dirichlet integral of $f$. Computing this area is known as the
area problem for the functions of type $f$. Thus, a function has a finite Dirichlet
integral exactly when its image has finite area (counting multiplicities). All
polynomials and, more generally, all functions $f\in \mathcal{A}$ for which $f'$ is bounded
on $\Delta$ are Dirichlet finite. Now by \eqref{Falpha}, \eqref{area} and since $\sum_{n=1}^{\infty}nr^{2(n-1)}=1/(1-r^2)^2$ we get:
\begin{corollary}
Let $0\leq \alpha< 1$. Then
\begin{equation*}
Area \, \{F_\alpha(\Delta)\}=\frac{\pi}{(1-\alpha^2)^2}.
\end{equation*}
\end{corollary}

Let $\mathcal{BS}(\alpha)$ be the subclass of $\mathcal{A}$ which satisfy the
condition
\begin{equation}\label{BS}
   \left(\frac{zf'(z)}{f(z)}-1\right)\prec F_{\alpha}(z)\qquad(z\in \Delta).
\end{equation}
The function class $\mathcal{BS}(\alpha)$ was studied extensively by Kargar $et$ $al$. \cite{KarEba}. The function
 \begin{equation}\label{ftilde}
    \tilde{f}(z) = z\left(\frac{1+z\sqrt{\alpha}}{1-z\sqrt{\alpha}} \right)^{\frac{1}{2\sqrt{\alpha}}},
  \end{equation}
is extremal function for several problems in the class $\mathcal{BS}(\alpha)$.
We note that the image of the function $F_{\alpha}(z)$ ($0\leq \alpha<1$) is the Booth lemniscate. We remark that a curve described by
\begin{equation*}\label{Blemn}
\left(x^2 + y^2\right)^2 - \left(n^4 + 2m^2\right)x^2 - \left(n^4 - 2m^2\right)y^2 = 0\qquad (x, y)\neq(0, 0),
\end{equation*}
(is a special case of Persian curve) was studied by Booth and is called the Booth lemniscate \cite{Booth}.
The Booth lemniscate is called elliptic if $n^4 > 2m^2$ while, for $n^4 < 2m^2$, it is termed hyperbolic. Thus
it is clear that the curve
\begin{equation*}
  \left(x^2+y^2\right)^2-\frac{x^2}{(1-\alpha)^2}-\frac{y^2}{(1+\alpha)^2}=0\qquad (x, y)\neq(0, 0),
\end{equation*}
is the Booth lemniscate of elliptic type. Thus the class $\mathcal{BS}(\alpha)$ is related to the Booth lemniscate.

In the present paper some properties of the class $\mathcal{BS}(\alpha)$ including, the order of strongly satarlikeness, upper and lower bound for $\mathfrak{Re} f(z)$, distortion and grow theorems and some sharp inequalities and logarithmic coefficients inequalities are considered. Also at the end, we introduce a certain subclass of convex functions.

\section{Main Results}
 Our first result is contained in the following. Further we recall that (see \cite{St}) the function $f$ is strongly starlike of order $\gamma$ and type $\beta$ in the disc $\Delta$, if it satisfies the following inequality:
\begin{equation}\label{stongly starlike}
  \left|\arg\left\{\frac{zf'(z)}{f(z)}-\beta\right\}\right|<\frac{\pi \gamma}{2}\qquad (0\leq \beta\leq 1, 0<\gamma\leq1).
\end{equation}
\begin{theorem}
  Let $0\leq \alpha \leq 1$ and $0<\varphi<2\pi$. If $f\in \mathcal{BS}(\alpha)$, then $f$ is strongly starlike function of order $\gamma(\alpha,\varphi)$ and type 1 where
  \begin{equation*}
    \gamma(\alpha,\varphi):=\frac{2}{\pi}\arctan \left(\frac{1+\alpha}{1-\alpha}|\tan \varphi|\right).
  \end{equation*}
\end{theorem}
\begin{proof}
  Let $z=re^{i\varphi}(r<1)$ and $\varphi\in (0,2\pi)$. Then we have
\begin{align*}
  F_\alpha(re^{i\varphi}) &=\frac{re^{i\varphi}}{1-\alpha r^2e^{2i\varphi}}.\frac{1-\alpha r^2e^{-2i\varphi}}{1-\alpha r^2e^{-2i\varphi}}  \\
   &= \frac{r(1-\alpha r^2)\cos\varphi+ir(1+\alpha r^2)\sin\varphi}{1-2\alpha r^2 \cos 2\varphi+\alpha^2 r^4}.
  \end{align*}
Hence
\begin{align}\label{phir}
  \left|\frac{\mathfrak{Im}\{F_\alpha(re^{i\varphi})\}}{\mathfrak{Re}\{F_\alpha(re^{i\varphi})\}}\right| &=
  \left|\frac{(1+\alpha r^2)\sin\varphi}{(1-\alpha r^2)\cos\varphi}\right| \nonumber\\
   &= \frac{1+\alpha r^2}{1-\alpha r^2}|\tan \varphi|\qquad (\varphi\in (0,2\pi)).
\end{align}
For such $r$ the curve $F_\alpha(re^{i\varphi})$ is univalent in $\Delta_r =\{z : |z| < r\}$. Therefore
\begin{equation}\label{[]Leftrightarrow[]}
  \left[\left(\frac{zf'(z)}{f(z)}-1\right)\prec F_\alpha(z),\ \ z\in \Delta_r\right]\Leftrightarrow
  \left[\left(\frac{zf'(z)}{f(z)}-1\right)\in F_\alpha(\Delta_r), \ \ z\in \Delta_r\right].
\end{equation}
Then by \eqref{phir} and \eqref{[]Leftrightarrow[]}, we have
  \begin{align*}
    \left|\arg\left\{\frac{zf'(z)}{f(z)}-1\right\}\right| &= \left|\arctan \frac{\mathfrak{Im}[(zf'(z)/f(z))-1]}{\mathfrak{Re}[(zf'(z)/f(z))-1]}\right|\\
     &\leq \left|\arctan \frac{\mathfrak{Im}(F_\alpha(re^{i\varphi}))}{\mathfrak{Re}(F_\alpha(re^{i\varphi}))}\right|\\
     &< \arctan\left(\frac{1+\alpha r^2}{1-\alpha r^2}|\tan\varphi|\right),
  \end{align*}
  and letting $r\rightarrow 1^-$, the proof of the theorem is completed.
\end{proof}

In the sequel we define an analytic function $\mathcal{L}(z)$ by
\begin{equation}\label{L}
  \mathcal{L}(z)=\exp \int_{0}^{z} \frac{1+F_\alpha(t)}{t} {\rm d}t\qquad (0\leq \alpha\leq 3-2\sqrt{2}, t\neq0),
\end{equation}
where $F_\alpha$ is given by \eqref{Falpha}. Since the function $F_\alpha$ is convex univalent for $0\leq \alpha\leq 3-2\sqrt{2}$, thus as result of (cf. \cite{ma-minda}), the function $\mathcal{L}(z)$ is convex univalent function in $\Delta$.
\begin{theorem}\label{th. L}
  Let $0\leq \alpha\leq 3-2\sqrt{2}$. If $f\in \mathcal{BS}(\alpha)$, then
  \begin{equation*}
    \mathcal{L}(-r)\leq \mathfrak{Re}\{f(z)\}\leq \mathcal{L}(r)\qquad (|z|=r<1),
  \end{equation*}
  where $\mathcal{L}(.)$ defined by \eqref{L}.
\end{theorem}

\begin{proof}
  Suppose that $f\in \mathcal{BS}(\alpha)$. Then by Lindel\"{o}f's principle of subordination (\cite{Good}), we get
  \begin{equation}\label{inf Re}
    \inf_{|z|\leq r} \mathfrak{Re}\{\mathcal{L}(z)\}\leq \inf_{|z|\leq r}\mathfrak{Re}\{f(z)\}\leq \sup_{|z|\leq r}\mathfrak{Re}\{f(z)\}\leq \sup_{|z|\leq r}\mathfrak{Re}\{|f(z)|\}\leq \sup_{|z|\leq r}\mathfrak{Re}\{\mathcal{L}(z)\}.
  \end{equation}
  Because $F_\alpha$ is a convex univalent function for $0\leq \alpha\leq 3-2\sqrt{2}$ and has real coefficients, hence $F_\alpha(\Delta)$ is a convex domain with respect to real axis. Moreover we have
  \begin{equation*}
    \sup_{|z|\leq r}\mathfrak{Re}\{\mathcal{L}(z)\}=\sup_{-r\leq z\leq r}\mathcal{L}(z)=\mathcal{L}(r)
  \end{equation*}
  and
    \begin{equation*}
    \inf_{|z|\leq r}\mathfrak{Re}\{\mathcal{L}(z)\}=\inf_{-r\leq z\leq r}\mathcal{L}(z)=\mathcal{L}(-r).
  \end{equation*}
 The proof of Theorem \ref{th. L} is thus completed.
\end{proof}

 \begin{theorem}
   Let $f\in \mathcal{BS}(\alpha)$, $0<\alpha \leq 3-2\sqrt{2}$, $r_s(\alpha)=\frac{\sqrt{1+4\alpha}-1}{2\alpha}\leq 0.8703$,
   \begin{equation*}
     F_\alpha(r_s(\alpha))=\max_{|z|=r_s(\alpha)<1}|F_\alpha(z)|\quad {and}\quad F_\alpha(-r_s(\alpha))=\min_{|z|=r_s(\alpha)<1}|F_\alpha(z)|.
   \end{equation*}
   Then we have
   \begin{equation}\label{|f'|}
     \frac{1}{1+r_s^2(\alpha)}(F_\alpha(r_s(\alpha))-1)\leq |f'(z)|\leq \frac{1}{1-r_s^2(\alpha)}(F_\alpha(r_s(\alpha))+1)
   \end{equation}
   and
   \begin{equation}\label{|f|}
     \int_{0}^{r_s(\alpha)}\frac{F_\alpha(t)}{1+t^2}{\rm d}t- \arctan r_s(\alpha)\leq |f(z)|\leq \frac{1}{2}\log \left(\frac{1+r_s(\alpha)}{1-r_s(\alpha)}\right)+\int_{0}^{r_s(\alpha)}\frac{F_\alpha(t)}{1-t^2}{\rm d}t
   \end{equation}
 \end{theorem}

\begin{proof}
  Let $f\in \mathcal{BS}(\alpha)$. Then by definition of subordination we have
  \begin{equation}\label{zf=1+F}
    \frac{zf'(z)}{f(z)}=1+F_\alpha(w(z)),
  \end{equation}
  where $w(z)$ is an analytic function $w(0)=0$ and $|w(z)|<1$. From \cite[Corollary 2.1]{KEL}, if $f\in \mathcal{BS}(\alpha)$, then $f$ is starlike univalent function in $|z|<r_s(\alpha)$, where $r_s(\alpha)=\frac{\sqrt{1+4\alpha}-1}{2\alpha}$. Thus if we define $q(z):\Delta_{r_s(\alpha)}\rightarrow \mathbb{C}$ by the equation $q(z):=f(z)$, where $\Delta_{r_s(\alpha)}:=\{z: |z|<r_s(\alpha)\}$, then $q(z)$ is starlike univalent function in $\Delta_{r_s(\alpha)}$ and therefore
  \begin{equation*}
    \frac{r_s(\alpha)}{1+r_s^2(\alpha)}\leq |q(z)|\leq \frac{r_s(\alpha)}{1-r_s^2(\alpha)}\qquad(|z|=r_s(\alpha)<1).
  \end{equation*}
  Now by \eqref{zf=1+F}, we have
  \begin{equation*}
    zf'(z)=q(z)(F_\alpha(z)+1)\qquad |z|=r_s(\alpha)<1.
  \end{equation*}
  Since $w(\Delta_{r_s(\alpha)})\subset \Delta_{r_s(\alpha)}$ and by the maximum principle for harmonic functions, we get
  \begin{align*}
    |f'(z)| &= \frac{|q(z)|}{|z|}|F_\alpha(w(z))+1| \\
    &\leq  \frac{1}{1-r_s^2(\alpha)}(|F_\alpha(w(z))|+1)\\
    &\leq \frac{1}{1-r_s^2(\alpha)}\left(\max_{|z|=r_s(\alpha)}|F_\alpha(w(z))|+1\right)\\
    &\leq \frac{1}{1-r_s^2(\alpha)}(F_\alpha(r_s(\alpha))+1).
  \end{align*}
  With the same proof we obtain
  \begin{equation*}
    |f'(z)|\geq \frac{1}{1+r_s^2(\alpha)}(F_\alpha(r_s(\alpha))-1).
  \end{equation*}
  Since the function
$f$ is a univalent function, the inequality for $|f(z)|$ follows from the
corresponding inequalities for $|f'(z)|$ by Privalov's Theorem \cite[Theorem 7, p. 67]{Good}.
\end{proof}

\begin{theorem}\label{t1}
  Let $F_\alpha(z)$ be given by \eqref{Falpha}. Then we have
  \begin{equation}\label{|F|}
 \frac{1}{1+\alpha}\leq \left|F_\alpha(z)\right|\leq \frac{1}{1-\alpha}\qquad (z\in \Delta-\{0\}, 0< \alpha<1).
  \end{equation}
\end{theorem}
\begin{proof}
  It is sufficient that to consider $\left|F_\alpha(z)\right|$ on the boundary
\begin{equation*}
  \partial F_\alpha(\Delta)=\left\{F_\alpha(e^{i\theta}): \theta\in [0,2\pi]\right\}.
\end{equation*}
 A simple check gives us
  \begin{equation}\label{x}
    x=\mathfrak{Re}\left\{F_\alpha(e^{i\theta})\right\}=\frac{(1-\alpha)\cos \theta}{1+\alpha^2-2\alpha\cos 2\theta}
  \end{equation}
  and
  \begin{equation}\label{y}
    y=\mathfrak{Im}\left\{F_\alpha(e^{i\theta})\right\}=\frac{(1+\alpha)\sin \theta}{1+\alpha^2-2\alpha\cos 2\theta}.
  \end{equation}
  Therefore, we have
  \begin{align}\label{|F|2}
    \left|F_\alpha (e^{i\theta})\right|^2&=\frac{1}{1+\alpha^2-2\alpha\cos 2\theta}\\
    &=\frac{1}{1+\alpha^2-2\alpha(2t^2-1)}=:H(t)\qquad(t=\cos\theta).
  \end{align}
  Since $0\leq t\leq1$, it is easy to see that $H'(t)\leq0$ when $-1\leq t\leq0$ and $H'(t)\geq0$ if $0\leq t\leq 1$.
 Thus
  \begin{equation*}
   \frac{1}{(1+\alpha)^2}\leq H(t)\leq \frac{1}{(1-\alpha)^2}\qquad (-1\leq t<0)
 \end{equation*}
 and
 \begin{equation*}
   \frac{1}{(1+\alpha)^2}\leq H(t)\leq \frac{1}{(1-\alpha)^2}\qquad (0< t\leq 1).
 \end{equation*}
  This completes the proof.
\end{proof}
A simple consequence of Theorem \ref{t1} as follows.

\begin{theorem}\label{co1}
  If $f\in \mathcal{BS}(\alpha)$ $(0< \alpha<1)$, then
  \begin{equation*}
  \frac{1}{1+\alpha}\leq  \left|\frac{zf'(z)}{f(z)}-1\right|\leq \frac{1}{1-\alpha}\qquad (z\in \Delta).
  \end{equation*}
The inequalities are sharp for the function $\tilde{f}$ defined by \eqref{ftilde}.
\end{theorem}
\begin{proof}
  By definition of subordination, and by using of Theorem \ref{t1}, the proof is obvious. For the sharpness of inequalities consider the function $\widetilde{f}$ which defined by $\eqref{ftilde}$. It is easy to see that
  \begin{equation*}
    \left|\frac{z\widetilde{f}'(z)}{\widetilde{f}(z)}-1\right|=\left|\frac{z}{1-\alpha z^2}\right|=|F_\alpha(z)|
  \end{equation*}
  and concluding the proof.
\end{proof}
The logarithmic coefficients $\gamma_n$ of $f(z)$ are defined by
\begin{equation}\label{log coef}
  \log\left\{\frac{f(z)}{z}\right\}=\sum_{n=1}^{\infty}2\gamma_n z^n\qquad (z\in \Delta).
\end{equation}
This coefficients play an important role for various estimates in the theory of univalent functions. For example, consider the Koebe function
\begin{equation*}
  k(z)=\frac{z}{(1-\mu z)^2}\qquad (\mu\in \mathbb{R}).
\end{equation*}
Easily seen that the above function $k(z)$ has logarithmic coefficients $\gamma_n(k)=\mu^n/n$ where $|\mu|=1$ and $n\geq1$. Also for $f\in \mathcal{S}$ we have
\begin{equation*}
  \gamma_1=\frac{a_2}{2}\quad{\rm and} \quad \gamma_2=\frac{1}{2}\left(a_3-\frac{a_2^2}{2}\right)
\end{equation*}
and the sharp estimates
\begin{equation*}
  |\gamma_1|\leq1 \quad{\rm and}\quad |\gamma_2|\leq \frac{1}{2}(1+2e^{-2})\approx 0.635\ldots,
\end{equation*}
hold. Also, sharp inequalities are known for sums involving logarithmic coefficients. For instance, the logarithmic coefficients $\gamma_n$ of every function $f\in \mathcal{S}$ satisfy the sharp inequality
\begin{equation}\label{ineq. pi26}
\sum_{n=1}^{\infty}|\gamma_n|^2\leq \frac{\pi^2}{6}
\end{equation}
and the equality is attained for the Koebe function (see \cite[Theorem 4]{DurLeu}).

The following lemma will be useful for the next result.
\begin{lemma}\label{lem.f(z)/z}{\rm (}see \cite[Theorem 2.1]{KarEba}{\rm )}
Let $f\in \mathcal{A}$ and $0\leq\alpha<1$. If $f\in
\mathcal{BS}(\alpha)$, then
\begin{equation}\label{11Th.f(z)/z}
    \log\frac{f(z)}{z}\prec\int_0^z \frac{P_{\alpha}(t)-1}{t}{\rm d}t\qquad (z\in\Delta),
\end{equation}
where
\begin{equation}\label{palpha}
    P_{\alpha}(z)-1=\frac{2}{\pi(1-\alpha)}i \log
    \left(\frac{1-e^{\pi
    i(1-\alpha)^2}z}{1-z}\right)\qquad (z\in\Delta)
\end{equation}
and
\begin{equation}\label{widP}
    \widetilde{P}_{\alpha}(z)=\int_0^z \frac{P_{\alpha}(t)-1}{t}{\rm d}t\qquad (z\in\Delta),
\end{equation}
are convex univalent in $\Delta$.
\end{lemma}

We remark that an analytic
function $P_{\mu,\beta}:\Delta\rightarrow \mathbb{C}$ by
\begin{equation}\label{P}
   P_{\mu,\beta}(z)=1+\frac{\beta-\mu}{\pi}i \log \left(\frac{1-e^{2\pi
   i\frac{1-\mu}{\beta-\mu}}z}{1-z}\right),\quad (\mu<1<\beta).
\end{equation}
 is a convex univalent function in $\Delta$, and has the form:
\begin{equation*}
  P_{\mu,\beta}(z)=1+\sum_{n=1}^{\infty} B_n z^n,
\end{equation*}
where
\begin{equation}\label{B-n}
B_n=\frac{\beta-\mu}{n\pi}i \left(1-e^{2n\pi
i\frac{1-\mu}{\beta-\mu}}\right),\qquad (n=1,2,\ldots).
\end{equation}

The above function $P_{\mu,\beta}(z)$ was introduced by Kuroki
and Owa \cite{KO2011} and they proved that $P_{\mu,\beta}$ maps
$\Delta$ onto a convex domain
\begin{equation}\label{Omega}
    P_{\mu,\beta}(\Delta)=\{ w\in \mathbb{C}: \mu<
    \mathfrak{Re}\{w\}<\beta\},
\end{equation}
conformally. Note that if we take $\mu=1/(\alpha-1)$ and $\beta=1/(1-\alpha)$ in \eqref{P}, then we have the function $P_\alpha$ which defined by \eqref{palpha}.
Now we have the following result about logarithmic coefficients.
\begin{theorem}
  Let $f\in \mathcal{A}$ belongs to the class $\mathcal{BS}(\alpha)$ and $0<\alpha<1$. Then the logarithmic coefficients of $f$ satisfy the inequality
  \begin{equation}\label{log ineq}
    \sum_{n=1}^{\infty}|\gamma_n|^2\leq\frac{1}{(1-\alpha)^2}\left[\frac{\pi^2}{45}-\frac{1}{\pi^2}
    \left(Li_4\left(e^{\pi(\alpha-2)i}\right)+Li_4\left(e^{\pi(2-\alpha)i}\right)\right)\right],
  \end{equation}
  where $Li_4$ is as following
  \begin{equation}\label{LI4}
    Li_4(z)=\sum_{n=1}^{\infty}\frac{z^n}{n^4}=-\frac{1}{2}\int_{0}^{1}\frac{\log^2(1/t)\log(1-tz)}{t}{\rm d}t.
  \end{equation}
  The inequality is sharp.
\end{theorem}
\begin{proof}
  If $f\in \mathcal{BS}(\alpha)$, then by using Lemma \ref{lem.f(z)/z} and with a simple calculation we get
  \begin{equation}\label{p1log}
    \log \frac{f(z)}{z}\prec\sum_{n=1}^{\infty}\frac{2}{\pi n^2(1-\alpha)}i\left(1-e^{\pi n (2-\alpha)i}\right)z^n\qquad(z\in \Delta).
  \end{equation}
  Now, by putting \eqref{log coef} into the last relation we have
  \begin{equation}\label{p2log}
        \sum_{n=1}^{\infty}2\gamma_n z^n\prec\sum_{n=1}^{\infty}\frac{1}{\pi n^2(1-\alpha)}i\left(1-e^{\pi n (2-\alpha)i}\right)z^n\qquad(z\in \Delta).
  \end{equation}
  Again, by Rogosinski's theorem \cite[6.2]{Dur}, we obtain
  \begin{align*}
    \sum_{n=1}^{\infty}|\gamma_n|^2 &\leq \sum_{n=1}^{\infty}\left|\frac{1}{\pi n^2(1-\alpha)}i\left(1-e^{\pi n (2-\alpha)i}\right)\right|^2 \\
     &= \frac{2}{\pi^2(1-\alpha)^2}\left(\sum_{n=1}^{\infty}\frac{1}{n^4}
     -\sum_{n=1}^{\infty}\frac{\cos\pi(2-\alpha)n}{n^4}\right)
  \end{align*}
  It is a simple exercise to verify that $\sum_{n=1}^{\infty}\frac{1}{n^4}=\pi^4/90$ and
  \begin{equation*}
    \sum_{n=1}^{\infty}\frac{\cos\pi(2-\alpha)n}{n^4}=\frac{1}{2}\left\{Li_4\left(e^{-i(2-\alpha)\pi
    }\right)+Li_4\left(e^{i(2-\alpha)\pi
    }\right)\right\}
  \end{equation*}
  and thus the desired inequality \eqref{log ineq} follows. For the sharpness of the inequality, consider
   \begin{equation}\label{sharp}
    F(z)=z\exp \widetilde{P}(z).
  \end{equation}
  It is easy to see that the function $F(z)$ belongs to the class $\mathcal{BS}(\alpha)$. Also, a simple check gives us
  \begin{equation*}
    \gamma_n(F(z))=\frac{1}{\pi n^2(1-\alpha)}i\left(1-e^{\pi n (2-\alpha)i}\right).
  \end{equation*}
Therefore the proof of this theorem is completed.
\end{proof}
\begin{theorem}
  Let $f\in \mathcal{BS}(\alpha)$. Then the logarithmic coefficients of $f$ satisfy
  \begin{equation*}
    |\gamma_n|\leq \frac{1}{2 n}\quad(n\geq1).
  \end{equation*}
\end{theorem}
\begin{proof}
  If $f\in \mathcal{BS}(\alpha)$, then by definition $\mathcal{BS}(\alpha)$, we have
  \begin{equation*}
    \frac{zf'(z)}{f(z)}-1=z\left(\log\left\{\frac{f(z)}{z}\right\}\right)'
    \prec F_{\alpha}(z).
\end{equation*}
Thus
\begin{equation*}
  \sum_{n=1}^{\infty}2 n \gamma_n z^n\prec \sum_{n=1}^{\infty}\alpha^{n-1}z^{2n-1}.
\end{equation*}
Applying the Rogosinski theorem \cite{Rog}, we get the inequality $2n |\gamma_n|\leq 1$. This completes the proof.
\end{proof}

\section{The class $\mathcal{BK}(\alpha)$}
In this section we introduce a new class.
Our principal definition is the following.
\begin{definition}
  Let $0\leq \alpha< 1$ and $F_\alpha$ be defined by \eqref{Falpha}. Then $f\in \mathcal{A}$ belongs to the class $\mathcal{BK}(\alpha)$ if $f$ satisfies the following:
  \begin{equation}\label{BK}
     \frac{zf''(z)}{f'(z)}\prec F_{\alpha}(z)\qquad(z\in \Delta).
  \end{equation}
\end{definition}

\begin{remark}
By Alexander's lemma $f\in \mathcal{BK}(\alpha)$, if and only if $zf'(z)\in \mathcal{BS}(\alpha)$.
Thus, if $f\in \mathcal{A}$ belongs to the class $\mathcal{BK}(\alpha)$, then
  \begin{equation*}
    \frac{\alpha}{\alpha-1}<\mathfrak{Re}\left\{1+\frac{zf''(z)}{f'(z)}\right\}<\frac{2-\alpha}{1-\alpha}\qquad(z\in \Delta).
  \end{equation*}
  \end{remark}

 The following theorem provides us a method of finding the members of the class $\mathcal{BK}(\alpha)$.
\begin{theorem}\label{BKiff}
  A function $f\in \mathcal{A}$ belongs to the class $\mathcal{BK}(\alpha)$ if and only if there exists a analytic function $q$, $q(z)\prec F_\alpha(z)$ such that
  \begin{equation}\label{fBK}
    f(z)=\int_{0}^{z}\left(\exp \int_{0}^{\zeta}\frac{q(t)}{t}\right)d\zeta.
  \end{equation}
\end{theorem}
\begin{proof}
  First, we let $f\in\mathcal{BK}(\alpha)$. Then from \eqref{BK} and by definition of subordination there exists a function $\omega\in \mathfrak{B}$ such that
  \begin{equation}\label{BKpt1}
    \frac{zf''(z)}{f'(z)}=F_\alpha(\omega(z))\qquad(z\in \Delta).
  \end{equation}
  Now we define $q(z)=F_\alpha(\omega(z))$ and so $q(z)\prec F_\alpha(z)$. The equation \eqref{BKpt1} readily gives
  \begin{equation*}
    \{\log f'(z)\}'=\frac{q(z)}{z}
  \end{equation*}
  and moreover
  \begin{equation*}
    f'(z)=\exp \left(\int_{0}^{\zeta}\frac{q(t)}{t}dt\right),
  \end{equation*}
  which upon integration yields \eqref{fBK}. Conversely, by simple calculations we see that if $f$ satisfies \eqref{fBK}, then $f\in \mathcal{BK}(\alpha)$ and therefore we omit the details.
\end{proof}

If we apply Theorem \ref{BKiff} with $q(z) = F_\alpha(z)$, then \eqref{fBK} with some easy calculations becomes
\begin{equation}\label{fhat}
  \hat{f}_\alpha(z):=z+\frac{z^2}{2}+\frac{1}{6}z^3+\frac{1}{12}\left(\alpha+\frac{1}{2}\right)z^4
  +\frac{1}{60}\left(4\alpha+\frac{1}{2}\right)z^5+\cdots.
\end{equation}
\begin{theorem}
  If a function $f(z)$ defined by \eqref{f} belongs to the class $\mathcal{BK}(\alpha)$, then
  \begin{equation*}
    |a_2|\leq \frac{1}{2} \quad {and }\quad |a_3|\leq \frac{1}{6}.
  \end{equation*}
  The equality occurs for $\hat{f}$ given in \eqref{fhat}.
\end{theorem}

\begin{proof}
  Assume that $f\in \mathcal{BK}(\alpha)$. Then from \eqref{BK} we have
  \begin{equation}\label{cof}
    \frac{zf''(z)}{f'(z)}=\frac{\omega(z)}{1-\alpha \omega^2(z)},
  \end{equation}
  where $\omega\in \mathfrak{B}$ and has the form $\omega(z)=b_1z+b_2 z^2+b_3 z^3+\cdots$. It is fairly well-known that if $|\omega(z)|=|b_1z+b_2 z^2+b_3 z^3+\cdots|<1$ $(z\in \Delta)$, then for all $k\in \mathbb{N}=\{1,2,3,\ldots\}$ we have $|b_k|\leq1$. Comparing the initial
coefficients in \eqref{cof} gives
\begin{equation}\label{a2a3}
  2a_2=b_1\quad {\rm and }\quad 6a_3-4a_2^2=b_2.
\end{equation}
  Thus $|a_2|\leq 1/2$ and $6a_3=b_1^2+b_2$. Since $|b_1|^2+|b_2|\leq 1$, therefore the assertion is obtained.
\end{proof}

\begin{corollary}
  It is well known that for $\omega(z)=b_1z+b_2 z^2+b_3 z^3+\cdots\in \mathfrak{B}$ for all $\mu\in \mathbb{C}$, we have $|b_2-\mu b_1^2|\leq \max \{1, |\mu|\}$. Therefore the Fekete-Szeg\"{o} inequality i.e. estimates of $|a_3-\mu a_2^2|$ for the class $\mathcal{BK}(\alpha)$ is equal to
  \begin{equation*}
    |a_3-\mu a_2^2|\leq \frac{1}{6}\max\left\{1,\left|\frac{3\mu}{2}-1\right|\right\}\qquad (\mu\in \mathbb{C}).
  \end{equation*}
\end{corollary}


{\bf Competing interests.}
The authors declare that they have no competing interests.

{\bf Authors' contributions.}
All authors of the manuscript have read and agreed to its content and are accountable for all aspects of the accuracy and integrity of the manuscript.

{\bf Acknowledgements.} The authors are thankful to the referee for the useful
suggestions.


\begin{thebibliography}{99}
\bibitem{Booth} J. Booth, {\it A Treatise on Some New Geometrical Methods}, Longmans, Green Reader
and Dyer, London, Vol. I (1873) and Vol. II (1877).
\bibitem{Dur} P.L. Duren, {\it Univalent functions,} Springer-Verlag, 1983.
\bibitem{DurLeu}P.L. Duren and Y.J. Leung, \textit{Logarithmic coefficients of univalent functions}, J. Anal. Math. {\bf36} (1979), 36--43 
\bibitem{Good} A.W. Goodman, \textit{Univalent Functions}, Vol.I and II, Mariner, Tampa, Florida, 1983.
\bibitem{KarEba} R. Kargar, A. Ebadian and J. Sok\'{o}{\l}, {\it On Booth lemiscate and starlike functions}, J. Anal. Math. Phys. (2017). https://doi.org/10.1007/s13324-017-0187-3
\bibitem{KEL} R. Kargar, A. Ebadian and L. Trojnar-Spelina, {\it Further results for starlike functions related with Booth lemniscate}, Iran. J. Sci. Technol. Trans. Sci. (accepted), arXiv:1802.03799.
\bibitem{KO2011} K. Kuroki and S. Owa,  Notes on New Class for Certain Analytic Functions, RIMS Kokyuroku Kyoto Univ. {\bf1772} (2011),  21--25.
\bibitem{Rog} W. Rogosinski, \textit{On the coefficients of subordinate functions},
Proc. London Math. Soc. {\bf48} (1943), 48--82.
\bibitem{ma-minda} W. Ma and D. Minda, {\it Uniformly convex functions}, Ann. Polon. Math. {\bf57} (1992) 165--175.
\bibitem{psok} K. Piejko and J. Sok\'{o}{\l}, {\it Hadamard product of analytic functions and
some special regions and curves}, J. Inequal. Appl. 2013, 2013:420.
\bibitem{Pijsok} K. Piejko and J. Sok\'{o}{\l}, {\it On Booth lemniscate and hadamard product of analytic functions}, Math. Slovaca {\bf65} (2015), 1337--1344.
\bibitem{St} Stankiewicz, J. \textit{Quelques probl\`{e}mes extr\'{e}maux dans les classes des fonctions $\alpha$-angulairement \'{e}toil\'{e}es}, Ann. Univ. Mariae Curie-Sk{\l}odowska Sect. A \textbf{20} (1966), 59--75.

\end{thebibliography}
\end{document}